\documentclass{amsart}

\usepackage{a4wide,latexsym,amssymb,amsthm,amsmath,color,setspace}

\theoremstyle{plain}

\newtheorem{theorem}{Theorem}
\newtheorem{lemma}{Lemma}

\newtheorem{definition}{Definition}

\theoremstyle{remark}

\newtheorem{remark}{Remark}
\newtheorem{example}{Example}

\def \n {\nabla}
\def \nt {\widetilde{\nabla}}
\def \Mt {\widetilde{M}}
\def \Rt {\widetilde{R}}

\def\<{ \left < }
\def\>{ \right > }

\def\R{\mathbb{R}}

\def\E{\mathbb{E}}
\def\H{\mathbb{H}}
\def\S{\mathbb{S}}

\def \Nil {\mathrm{Nil}_3}

\begin{document}

\title[Constant angle surfaces]
{Constant angle surfaces in the Heisenberg group}

\author[J.~Fastenakels]{Johan Fastenakels}
\thanks{The first named author is a research assistant of the Research Foundation - Flanders (FWO)}
\author[M.~I.~Munteanu]{Marian Ioan Munteanu}
\thanks{The second named author was supported by Grant PN-II ID 398/2007-2010 (Romania)}
\author[J.~Van der Veken]{Joeri Van der Veken}
\thanks{The third named author is a postdoctoral researcher supported by the Research Foundation - Flanders (FWO)}

\address{Katholieke Universiteit Leuven\\ Celestijnenlaan 200 B\\ B-3001 Leuven\\ Belgium}
\email[J.~Fastenakels]{johan.fastenakels@wis.kuleuven.be}
\email[J.~Van der Veken]{joeri.vanderveken@wis.kuleuven.be}

\address{``Al.I.Cuza'' University of Iasi\\ Bd. Carol I, n. 11\\
700506 - Iasi\\ Romania}
\email[M.~I.~Munteanu]{marian.ioan.munteanu@gmail.com}

\begin{abstract}

In this article we generalize the notion of constant angle surfaces
in $\mathbb{S}^2\times\mathbb{R}$ and $\mathbb{H}^2\times\mathbb{R}$
to general Bianchi-Cartan-Vranceanu spaces. We show that these
surfaces have constant Gaussian curvature and we give a complete
local classification in the Heisenberg group.

\end{abstract}



\maketitle

\section{Introduction}

Differential geometry of submanifolds started with the study of
surfaces in the 3-dimensional Euclidean space. This study has been
generalized in two ways. One can generalize the dimension and the
codimension of the submanifold, but one can also generalize the
ambient space. The most popular ambient spaces are, besides the
Euclidean space, the sphere and the hyperbolic space. These spaces
are called real space forms and are the easiest Riemannian manifolds
from geometric point of view. With the work of, among others,
Thurston, 3-dimensional homogeneous spaces have gained interest. A
homogeneous space is a Riemannian manifold such that for every two
points $p$ and $q$, there exists an isometry mapping $p$ to $q$.
Roughly speaking, this means that the space looks the same at every
point.

There are three classes of 3-dimensional homogeneous spaces
depending on the dimension of the isometry group. This dimension can
be 3, 4 or 6. In the simply connected case, dimension 6 corresponds
to one of the three real space forms and dimension 3 to a general
simply connected 3-dimensional Lie group with a left-invariant
metric. In this article we restrict ourselves to dimension 4. A
3-dimensional homogeneous spaces with a 4-dimensional isometry group
is locally isometric to (an open part of) $\R^3$, equipped with a
metric depending on two real parameters. Since this two-parameter
family of metrics first appeared in the works of Bianchi, Cartan and
Vranceanu, these spaces are often referred to as
Bianchi-Cartan-Vranceanu spaces, or BCV-spaces for short. More
details can be found in the last section. Some well-known examples
of BCV-spaces are the Riemannian product spaces
$\mathbb{S}^2\times\mathbb{R}$ and $\mathbb{H}^2\times\mathbb{R}$
and the 3-dimensional Heisenberg group.

There exists a Riemannian submersion from a general BCV-space onto a
surface of constant Gaussian curvature. Since this submersion
extends the classical Hopf fibration
$\pi:\S^3(\kappa/4)\rightarrow\S^2(\kappa)$, it is also called the
Hopf fibration. Hence, any BCV-space is foliated by 1-dimensional
fibers of the Hopf fibration. Now consider a surface immersed in a
BCV-space and consider at every point of the surface the angle
between the unit normal and the fiber of the Hopf fibration through
this point. The existence and uniqueness theorem for immersions into
BCV-spaces, proven in \cite{D}, shows that this angle function is
one of the fundamental invariants for a surface in a BCV-space.
Hence, it is a very natural problem to look for those surfaces for
which this angle function is constant.

The constant angle surfaces in $\S^2 \times \R$ and $\H^2 \times \R$
have been classified in \cite{DFVdVV} and \cite{DM1}, see also
\cite{DM2}. In these cases, the Hopf fibration is just the natural
projection $\pi : \S^2 \times \R \rightarrow \S^2$, respectively
$\pi : \H^2 \times \R \rightarrow \H^2$. For an overview of constant
angle surfaces in $\E^3$, i.e., surfaces for which the unit normal
makes a constant angle with a fixed direction in $\E^3$, we refer to
\cite{MN}. In the present article, we show that all constant angle
surfaces in any BCV-space have constant Gaussian curvature and we
give a complete local classification of constant angle surfaces in
the Heisenberg group by means of an explicit parametrization. We
also give some partial results for a classification in general
BCV-spaces.

\section{Preliminaries}

\subsection{The Heisenberg group}

Let $(V,\omega)$ be a symplectic vector space of dimension $2n$.
Then the associated Heisenberg group is defined as the set
$V\times\R$ equipped with the operation
$$(v_1,t_1)\ast (v_2,t_2)=\left(v_1+v_2, t_1+t_2+\frac{1}{2}\omega(v_1,v_2)\right).$$
From now on, we restrict ourselves to the 3-dimensional Heisenberg
group coming from $\R^2$ with the canonical symplectic form
$\omega((x,y),(\overline{x},\overline{y}))=x\overline{y}-\overline{x}y$,
i.e., we consider $\R^3$ with the group operation
$$(x,y,z)\ast(\overline{x},\overline{y},\overline{z})=\left(x+\overline{x},\ y+\overline{y},\ z+\overline{z}+\frac{x\overline{y}}{2}-\frac{\overline{x}y}{2}\right).$$
Remark that the mapping
$$\R^3\rightarrow\left\{\left.\left(\begin{array}{ccc}1&a&b\\0&1&c\\0&0&1\end{array}\right)\ \right|\ a,b,c\in\R\right\} : (x,y,z)\mapsto\left(\begin{array}{ccc}1&x&z+\frac{xy}{2}\\0&1&y\\0&0&1\end{array}\right)$$
is an isomorphism between $(\R^3,\ast)$ and a subgroup of
$\mathrm{GL}(3,\R)$. For every non-zero real number $\tau$ the
following Riemannian metric on $(\R^3,\ast)$ is left-invariant:
$$ds^2=dx^2+dy^2+4\tau^2\left(dz+\frac{y\,dx-x\,dy}{2}\right)^2.$$
After the change of coordinates $(x,y,2\tau z)\mapsto(x,y,z)$, this
metric is expressed as
\begin{equation} \label{metricNil}
ds^2=dx^2+dy^2+\left(dz+\tau(y\,dx-x\,dy)\right)^2.
\end{equation}
From now on, we denote by $\Nil$ the group $(\R^3,\ast)$ with the
metric \eqref{metricNil}. By some authors, the notation $\Nil$ is
only used if $\tau=\frac12$.

In the following lemma, we give a left-invariant orthonormal frame
on $\Nil$ and describe the Levi Civita connection and the Riemann
Christoffel curvature tensor in terms of this frame.
\begin{lemma}\label{Lem1}
The following vector fields form a left-invariant orthonormal frame
on $\Nil$:
$$e_1 = \partial_x-\tau y\partial_z,\quad
e_2 = \partial_y+\tau x\partial_z,\quad e_3 = \partial_z.$$ The
geometry of $\Nil$ can be described in terms of this frame as
follows.
\begin{itemize}
\item[(i)] These vector fields satisfy the commutation relations
$$[e_1,e_2]= 2\tau e_3, \qquad [e_2,e_3]=0, \qquad [e_3,e_1]=0.$$
\item[(ii)] The Levi Civita connection $\nt$ of $\Nil$ is given by
$$\begin{array}{lll} \widetilde{\nabla}_{e_1}e_1=0, & \widetilde{\nabla}_{e_1}e_2=\tau e_3, & \widetilde{\nabla}_{e_1}e_3=-\tau e_2, \\
                     \widetilde{\nabla}_{e_2}e_1=-\tau e_3, & \widetilde{\nabla}_{e_2}e_2=0, & \widetilde{\nabla}_{e_2}e_3=\tau e_1, \\
                     \widetilde{\nabla}_{e_3}e_1=-\tau e_2, & \widetilde{\nabla}_{e_3}e_2=\tau e_1, & \widetilde{\nabla}_{e_3}e_3=0.
\end{array}$$
\item[(iii)] The Riemann Christoffel curvature tensor $\Rt$ of
$\Nil$ is determined by
\begin{eqnarray*}
\widetilde{R}(X,Y)Z &=& -3\tau^2(\langle Y,Z\rangle X-\langle X,Z\rangle Y)\\
&& +4\tau^2(\langle Y,e_3 \rangle\langle Z,e_3\rangle X - \langle X,e_3 \rangle\langle Z,e_3 \rangle Y\\
&& +\langle X,e_3 \rangle\langle Y,Z \rangle e_3 - \langle Y,e_3
\rangle\langle X,Z \rangle e_3)
\end{eqnarray*}
for $p\in \Nil$ and $X,Y,Z\in T_p\Nil$.
\end{itemize}
\end{lemma}

It is clear that the Killing vector field $e_3$ plays an important
role in the geometry of $\Nil$. In fact, the integral curves of
$e_3$ are the fibers of the Hopf fibration:
\begin{definition}
The mapping $\Nil\to\E^2:(x,y,z)\mapsto(x,y)$ is a Riemannian
submersion, called the \emph{Hopf fibration}. The inverse image of a
curve in $\E^2$ under the Hopf fibration is called a \emph{Hopf
cylinder}.
\end{definition}

\subsection{Surfaces in the Heisenberg group}

Let $F:M\rightarrow\Nil$ be an isometric immersion of an oriented
surface in the Heisenberg group. Denote by $N$ a unit normal vector
field, by $S$ the associated shape operator, by $\theta$ the angle
between $e_3$ and $N$ and by $T$ the projection of $e_3$ onto the
tangent plane to $M$, i.e. the vector field $T$ on $M$ such that
$F_{\ast}T+\cos\theta\,N=e_3$.

If we work locally, we may assume $\theta\in [0,\frac{\pi}2]$. Take
$p\in M$ and $X,Y,Z\in T_pM$. Then the equations of Gauss and
Codazzi are given respectively by
\begin{eqnarray}
R(X,Y)Z &=& -3\tau^2(\langle Y,Z\rangle X-\langle X,Z\rangle Y)\label{GaussNil}\\
        & & +4\tau^2(\langle Y,T \rangle\langle Z,T \rangle X - \langle X,T \rangle\langle Z,T \rangle Y\nonumber\\
        & & +\langle X,T \rangle\langle Y,Z \rangle T - \langle Y,T \rangle\langle X,Z \rangle T)\nonumber\\
        & & +\langle SY,Z \rangle SX - \langle SX,Z \rangle SY
        \nonumber
\end{eqnarray}
and
\begin{equation}\label{CodazziNil}
\nabla_XSY-\nabla_YSX-S[X,Y]=-4\tau^2\cos\theta(\langle Y,T\rangle
X-\langle X,T\rangle Y).
\end{equation}
Remark that (\ref{GaussNil}) is equivalent to the following relation
between the Gaussian curvature $K$ of the surface and extrinsic data
of the immersion:
\begin{equation}
K=\det S+\tau^2-4\tau^2\cos^2\theta.\label{Gauss2Nil}
\end{equation}
Finally, we remark that also the following structure equations hold:
\begin{eqnarray}
\nabla_XT &=& \cos\theta(SX-\tau JX),\label{Struct1Nil}\\
X[\cos\theta] &=& -\langle SX-\tau JX,T\rangle,\label{Struct2Nil}
\end{eqnarray}
where $JX$ is defined by $JX:=N\times X$, i.e., $J$ denotes the
rotation over $\frac{\pi}2$ in $T_pM$.

The four equations \eqref{CodazziNil}, \eqref{Gauss2Nil},
\eqref{Struct1Nil} and \eqref{Struct2Nil} are called the
compatibility equations since it was proven in \cite{D} that they
are necessary and sufficient conditions to have an isometric
immersion of a surface into $\Nil$.

\section{First results on constant angle surfaces in $\Nil$}

In this section we define constant angle surfaces in $\Nil$ and we
give some immediate consequences of this assumption.

\begin{definition}
We say that a surface in the Heisenberg group $\Nil$ is a
\emph{constant angle surface} if the angle $\theta$ between the unit
normal and the direction $e_3$ tangent to the fibers of the Hopf
fibration is the same at every point.
\end{definition}

\begin{lemma}\label{Lem2}
Let $M$ be a constant angle surface in $\Nil$. Then the following
statements hold.
\begin{itemize}
\item[(i)] With respect to the basis $\{T,JT\}$, the shape operator is
given by $$S=\left(\begin{array}{cc}0 & -\tau \\ -\tau & \lambda
\end{array}\right)$$ for some function $\lambda$ on $M$.
\item[(ii)] The Levi Civita connection of the surface is determined by
$$\begin{array}{ll}\n_TT=-2\tau\cos\theta\,JT, &
\n_{JT}T=\lambda\cos\theta\,JT,\\ \n_TJT=2\tau\cos\theta\,T, &
\n_{JT}JT=-\lambda\cos\theta\,T.\end{array}$$
\item[(iii)] The Gaussian curvature of the surface is a negative constant given by
$$K=-4\tau^2\cos^2\theta.$$
\item[(iv)] The function $\lambda$ satisfies the following PDE:
$$T[\lambda]+\lambda^2\cos\theta+4\tau^2\cos^3\theta=0.$$
\end{itemize}
\end{lemma}

\begin{proof} The first
statement follows from equation \eqref{Struct2Nil} and the symmetry
of the shape operator. The second statement follows from equation
\eqref{Struct1Nil}, (i) and the equalities
$\<T,T\>=\<JT,JT\>=\sin^2\theta$ and $\<T,JT\>=0$. The third
statement follows from the equation of Gauss \eqref{Gauss2Nil} and
(i). The last statement follows from (ii) and (iii), or from the
equation of Codazzi \eqref{CodazziNil} and (i).
\end{proof}

\section{Classification of constant angle surfaces in $\Nil$}

By using Lemma \ref{Lem2}, we can prove our main result, namely the
complete local classification of all constant angle surfaces in the
Heisenberg group.

\begin{theorem}\label{Theo1}
Let $M$ be a constant angle surface in the Heisenberg group $\Nil$.
Then $M$ is isometric to an open part of one of the following types
of surfaces:
\begin{itemize}
\item[(i)] a Hopf-cylinder,
\item[(ii)] a surface given by
\begin{eqnarray*}
F(u,v) &=& \left(\frac{1}{2\tau}\tan\theta\sin u+f_1(v),\
-\frac{1}{2\tau}\tan\theta\cos u+f_2(v),\right.\\
&& \left.-\frac{1}{4\tau}\tan^2\theta\,u-\frac{1}{2}\tan\theta\cos
uf_1(v)-\frac{1}{2}\tan\theta\sin uf_2(v)-\tau f_3(v)\right),
\end{eqnarray*}
with $(f_1')^2+(f_2')^2=\sin^2\theta$ and
$f'_3(v)=f'_1(v)f_2(v)-f_1(v)f'_2(v)$. Here, $\theta$ denotes the
constant angle.
\end{itemize}
\end{theorem}

\begin{proof}
Let $M$ be a constant angle surface in $\Nil$. If
$\theta=\frac{\pi}2$, the surface is of the first type mentioned in
the theorem. Moreover, $\theta=0$ gives a contradiction with Lemma
\ref{Lem1} (i), since $\tau\neq 0$. From now on, we assume that
$\theta$ lies strictly between $0$ and $\frac{\pi}2$. There exists a
locally defined real function $\phi$ on $M$ such that the unit
normal on $M$ is given by
\begin{equation}
\label{eq1} N=\sin\theta\cos\phi\,e_1+\sin\theta\sin\phi\,
e_2+\cos\theta\,e_3.
\end{equation}
Remark that with this notation, we have
\begin{align}
& T = -\sin\theta(\cos\theta\cos\phi\,e_1+\cos\theta\sin\phi\,e_2-\sin\theta\,e_3),\label{eq2}\\
& JT = \sin\theta(\sin\phi\,e_1-\cos\phi\,e_2).\label{eq3}
\end{align}

By a straightforward computation, using Lemma \ref{Lem1} (ii),
\eqref{eq1}, \eqref{eq2} and \eqref{eq3}, we obtain that the shape
operator $S$ satisfies
\begin{align*}
& ST=-\nt_TN=(T[\phi]-\tau\sin^2\theta+\tau\cos^2\theta)JT,\\
& SJT =-\nt_{JT}N=-\tau T+(JT)[\phi]JT.
\end{align*}
Comparing this to Lemma \ref{Lem2} (i), gives
\begin{equation}\label{eq4}
\left\{\begin{array}{l} T[\phi] = -2\tau\cos^2\theta, \\
(JT)[\phi] = \lambda. \end{array}\right.
\end{equation}
The integrability condition for this system of equations is
precisely the PDE from Lemma \ref{Lem2} (iv). Remark that the
solutions of this system describe the possible tangent distributions
to $M$.

In order to solve the system \eqref{eq4}, let us choose coordinates
$(u,v)$ on $M$ such that $\partial_u=T$ and $\partial_v=aT+bJT$, for
some real functions $a$ and $b$ which are locally defined on $M$.
The condition $[\partial_u,\partial_v]=0$ is equivalent to the
following system of equations:
\begin{equation}
\label{eq5}
\left\{\begin{array}{l} \partial_ua=-2\tau b\cos\theta, \\
\partial_ub=\lambda b\cos\theta. \end{array}\right.
\end{equation}
The PDE from Lemma \ref{Lem2} (iv) is now equivalent to
$\partial_u\lambda+\lambda^2\cos\theta+4\tau^2\cos^3\theta=0,$ for
which the general solution is given by
\begin{equation*}
\lambda(u,v)=2\tau\cos\theta\tan(\varphi(v)-2\tau\cos^2\theta\,u),
\end{equation*}
for some function $\varphi(v)$. We can now solve system \eqref{eq5}.
Remark that we are interested in only one coordinate system on the
surface $M$ and hence we only need one solution for $a$ and $b$, for
example
\begin{align}
&a(u,v)=\frac{1}{\cos\theta}\sin(\varphi(v)-2\tau\cos^2\theta\,u),\label{eq6}\\
& b(u,v)=\cos(\varphi(v)-2\tau\cos^2\theta\,u).\label{eq7}
\end{align}
System \eqref{eq4} is then equivalent to
$$
\left\{\begin{array}{l} \partial_u\phi = -2\tau\cos^2\theta, \\
\partial_v\phi = 0, \end{array}\right.
$$
for which the general solution is given by
\begin{equation}\phi(u,v)=\phi(u)=-2\tau\cos^2\theta\,u+c,\label{eq8}\end{equation}
where $c$ is a real constant.

To finish the proof, we have to integrate the distribution spanned
by $T$ and $JT$. Denote the parametrization of $M$ by
$$F:U\subseteq\R^2\rightarrow M\subset\Nil:(u,v)\mapsto
F(u,v)=(F_1(u,v),F_2(u,v),F_3(u,v)).$$ Then we know from
\eqref{eq2}, \eqref{eq3} and the choice of coordinates $(u,v)$ that
\begin{align*}
& (\partial_u F_1, \partial_u F_2, \partial_u F_3) = -\sin\theta(\cos\theta\cos\phi\,e_1+\cos\theta\sin\phi\,e_2-\sin\theta\,e_3),\\
& (\partial_v F_1, \partial_v F_2, \partial_v F_3) =
\sin\theta((-a\cos\theta\cos\phi+b\sin\phi)e_1 -
(a\cos\theta\sin\phi+b\cos\phi)e_2+a\sin\theta\,e_3).
\end{align*}
Moreover, at the point $(F_1,F_2,F_3)$, we have $e_1=(1,0,-\tau
F_2)$, $e_2=(0,1,\tau F_1)$, $e_3=(0,0,1)$ and $a$, $b$ and $\phi$
are given by \eqref{eq6}, \eqref{eq7} and \eqref{eq8}. Direct
integration, followed by the reparametrization
$-2\tau\cos^2\theta\,u+c\mapsto u$ yields the parametrization given
in the theorem, where $f_1(v)$ and $f_2(v)$ are primitive functions
of $\sin\theta\sin(c-\varphi(v))$ and
$\sin\theta\cos\left(c-\varphi(v)\right)$ respectively.
\end{proof}

\begin{remark}
We may assume that $f_1(v)$ and $f_2(v)$ are polynomials of degree
at most one, by changing the $v$-coordinate if necessary. $f_3(v)$
is then a polynomial of degree at most two.
\end{remark}

We end this section with a concrete example of a non-trivial
constant angle surface in $\Nil$.
\begin{example}
Take $f_1(v)=f_3(v)=0$ and $f_2(v)=\frac{1}{\sqrt{2}}v$. Then it
follows from Theorem \ref{Theo1} that the surface
$$F(u,v)=\left(\frac{1}{2\tau}\sin u, -\frac{1}{2\tau}\cos u + \frac{1}{\sqrt{2}}v, -\frac{1}{4\tau}u-\frac{1}{2\sqrt{2}}v\sin
u\right)$$ is a constant angle surface in $\Nil$, with
$\theta=\frac{\pi}4$. This surface is a ruled surface based on a
helix.
\end{example}

\section{A possible generalization to BCV-spaces}

\subsection{Bianchi-Cartan-Vranceanu spaces}

Constant angle surfaces have now been classified in the homogeneous
3-spaces $\E^3$, $\mathbb{S}^2\times\mathbb{R}$,
$\mathbb{H}^2\times\mathbb{R}$ and $\Nil$. All of these spaces
belong to a larger family, called the BCV-spaces.

\begin{definition} Let $\kappa$ and $\tau$ be real numbers, with $\tau\geq 0$.
The \emph{Bianchi-Cartan-Vranceanu space (BCV-space)}
$\Mt^3(\kappa,\tau)$ is defined as the set
$$\left\{(x,y,z)\in\R^3\ \left|\ 1+\frac{\kappa}{4}(x^2+y^2)>0\right.\right\}$$
equipped with the metric
$$ds^2=\frac{dx^2+dy^2}{(1+\frac{\kappa}{4}(x^2+y^2))^2}+\left(dz+\tau\frac{y\,dx-x\,dy}{1+\frac{\kappa}{4}(x^2+y^2)}\right)^2.$$
\end{definition}

It was Cartan who obtained this family of spaces, with $\kappa\neq
4\tau^2$, as the result of the classification of 3-dimensional
Riemannian manifolds with 4-dimensional isometry group. They also
appear in the work of Bianchi and Vranceanu. The complete
classification of BCV-spaces is as follows:
\begin{itemize}
\item if $\kappa=\tau=0$, then $\Mt^3(\kappa,\tau)\cong \E^3$,
\item if $\kappa=4\tau^2\neq 0$, then $\Mt^3(\kappa,\tau)\cong \S^3(\kappa/4)\setminus\{\infty\}$,
\item if $\kappa >0$ and $\tau=0$, then $\Mt^3(\kappa,\tau)\cong (\S^2(\kappa)\setminus\{\infty\})\times\R$,
\item if $\kappa <0$ and $\tau=0$, then $\Mt^3(\kappa,\tau)\cong \H^2(\kappa)\times\R$,
\item if $\kappa >0$ and $\tau\neq 0$, then $\Mt^3(\kappa,\tau)\cong\mathrm{SU}(2)\setminus\{\infty\}$,
\item if $\kappa <0$ and $\tau\neq 0$, then $\Mt^3(\kappa,\tau)\cong\widetilde{\mathrm{SL}}(2,\R)$,
\item if $\kappa =0$ and $\tau\neq 0$, then $\Mt^3(\kappa,\tau)\cong \Nil$.
\end{itemize}
The Lie groups in the classification are equipped with a standard
left-invariant metric. This classification implies that
$\Mt^3(\kappa,\tau)$ has constant sectional curvature if and only if
$\kappa=4\tau^2$. The curvature is then equal to $\frac \kappa
4=\tau^2\geq 0$. The following lemma generalizes Lemma \ref{Lem1}.

\begin{lemma}\label{Lem3}
The following vector fields form an orthonormal frame on
$\Mt^3(\kappa,\tau)$: $$e_1 =
\left(1+\frac{\kappa}{4}(x^2+y^2)\right)\partial_x-\tau
y\partial_z,\quad e_2 =
\left(1+\frac{\kappa}{4}(x^2+y^2)\right)\partial_y+\tau
x\partial_z,\quad e_3 = \partial_z.$$ The geometry of the BCV-space
can be described in terms of this frame as follows.
\begin{itemize}
\item[(i)] These vector fields satisfy the commutation relations
$$[e_1,e_2]=-\frac{\kappa}{2}ye_1+\frac{\kappa}{2}xe_2+2\tau e_3,
\qquad [e_2,e_3]=0, \qquad [e_3,e_1]=0.$$
\item[(ii)] The Levi Civita connection $\nt$ of $\Mt^3(\kappa,\tau)$ is given
by
$$\begin{array}{lll} \widetilde{\nabla}_{e_1}e_1=\frac{\kappa}{2}ye_2, & \widetilde{\nabla}_{e_1}e_2=-\frac{\kappa}{2}ye_1+\tau e_3, & \widetilde{\nabla}_{e_1}e_3=-\tau e_2, \\
                   \widetilde{\nabla}_{e_2}e_1=-\frac{\kappa}{2}xe_2-\tau e_3, & \widetilde{\nabla}_{e_2}e_2=\frac{\kappa}{2}xe_1, & \widetilde{\nabla}_{e_2}e_3=\tau e_1, \\
                   \widetilde{\nabla}_{e_3}e_1=-\tau e_2, & \widetilde{\nabla}_{e_3}e_2=\tau e_1, & \widetilde{\nabla}_{e_3}e_3=0.
                   \end{array}$$
\item[(iii)] The Riemann Christoffel  curvature tensor $\Rt$ of
$\Mt^3(\kappa,\tau)$ is determined by
\begin{eqnarray*}
\widetilde{R}(X,Y)Z &=& (\kappa-3\tau^2)(\langle Y,Z\rangle X-\langle X,Z\rangle Y)\\
&& -(\kappa-4\tau^2)(\langle Y,e_3 \rangle\langle Z,e_3\rangle X - \langle X,e_3 \rangle\langle Z,e_3 \rangle Y\\
&& + \langle X,e_3 \rangle\langle Y,Z \rangle e_3 - \langle Y,e_3
\rangle\langle X,Z \rangle e_3),
\end{eqnarray*}
for $p\in \Mt^3(\kappa,\tau)$ and $X,Y,Z\in T_p\Mt^3(\kappa,\tau)$.
\end{itemize}
\end{lemma}
Also the notions of Hopf fibration and Hopf cylinder can be
generalized to BCV-spaces.
\begin{definition} Let $\Mt^2(\kappa)$ be the following Riemannian surface
of constant Gaussian curvature $\kappa$:
$$\Mt^2(\kappa)=\left(\left\{(x,y)\in\R^2\ \left|\ 1+\frac{\kappa}{4}(x^2+y^2)>0\right.\right\}\ ,\ \frac{dx^2+dy^2}{(1+\frac{\kappa}{4}(x^2+y^2))^2}\right).$$
Then the mapping
$\pi:\Mt^3(\kappa,\tau)\rightarrow\Mt^2(\kappa):(x,y,z)\mapsto
(x,y)$ is a Riemannian submersion, called the \emph{Hopf-fibration}.
The inverse image of a curve in $\Mt^2(\kappa)$ under $\pi$ is
called a \emph{Hopf-cylinder} and by a \emph{leaf} of the
Hopf-fibration one denotes a surface which is everywhere orthogonal
to the fibres of $\pi$.\end{definition}

Remark that in the special case that $\kappa=4\tau^2\neq 0$, $\pi$
coincides locally with the classical Hopf-fibration
$\pi:\S^3\left(\kappa/4\right)\rightarrow\S^2(\kappa)$. From the
theorem of Frobenius and Lemma \ref{Lem3} (i), it is clear that
leaves of $\pi$ only exist if $\tau=0$. They are nothing but open
parts of surfaces of type $\E^2\times\{t_0\}$,
$(\S^2(\kappa)\setminus\{\infty\})\times\{t_0\}$ or
$\H^2(\kappa)\times\{t_0\}$.

\subsection{Constant angle surfaces in BCV-spaces}
Let $F:M\rightarrow\Mt^3(\kappa,\tau)$ be an isometric immersion of
an oriented surface in a BCV-space. We can define $N$, $S$, $\theta$
and $T$ as in the case of $\Nil$. The four compatibility equations
become
\begin{align*}
& \nabla_XSY-\nabla_YSX-S[X,Y]=(\kappa-4\tau^2)\cos\theta(\langle Y,T\rangle X-\langle X,T\rangle Y),\\
& K=\det S+\tau^2+(\kappa-4\tau^2)\cos^2\theta,\\
& \nabla_XT = \cos\theta(SX-\tau JX),\\
& X[\cos\theta] = -\langle SX-\tau JX,T\rangle.
\end{align*}

One can define constant angle surfaces in a general BCV-space in the
same way as in the Heisenberg group. In the special case that the
BCV-space is $\S^2\times\R$ or $\H^2\times\R$, this definition
coincides with the ones used in \cite{DFVdVV} and \cite{DM1}. Also
Lemma \ref{Lem2} can be easily adapted.

\begin{lemma}\label{Lem4}
Let $M$ be a constant angle surface in a general BCV-space
$\widetilde{M}(\kappa,\tau)$. Then the following statements hold.
\begin{itemize}
\item[(i)] The shape operator with respect to $\{T,JT\}$ is the
same as in Lemma \emph{\ref{Lem2} (i)}.
\item[(ii)] The Levi Civita connection of the surface is determined by the same formulas as in Lemma \emph{\ref{Lem2}
(ii)}.
\item[(iii)] The Gaussian curvature of the surface is a constant given by
$$K=(\kappa-4\tau^2)\cos^2\theta.$$
\item[(iv)] The function $\lambda$ satisfies the following PDE:
$$T[\lambda]+\lambda^2\cos\theta+\kappa\cos\theta\sin^2\theta+4\tau^2\cos^3\theta=0.$$
\end{itemize}
\end{lemma}

In a further study of constant angle surfaces in BCV-spaces, one may
assume that $\kappa\neq 0$ and $\tau\neq 0$, since in all other
cases there is now a complete classification. We see that $\theta=0$
cannot occur, since the distribution spanned by $e_1$ and $e_2$ is
not integrable, and that $\theta=\frac{\pi}{2}$ gives a Hopf
cylinder.

In the final remark, we give some partial results that may lead to a
complete classification.

\begin{remark}
To get an explicit local classification for a constant angle surface
with $\theta\in] 0,\frac{\pi}{2}[$ in a BCV-space with
$\kappa\tau\neq 0$, we solve the PDE for $\lambda$, given in Lemma
\ref{Lem4} (iv), by choosing local coordinates $(u,v)$ on the
surface, for instance such that $T=\partial_u$ and
$\partial_v=aT+bJT$. This solution depends on the sign of
$\kappa\sin^2\theta+4\tau^2\cos^2\theta$. If we suppose that we are
in the case that this is strictly positive and denote this constant
by $r^2$, then we find
\begin{align*}
& \lambda(u,v)=r\tan(\varphi(v)-r\cos\theta u),\\
& a(u,v)=\frac{2\tau}r\sin(\varphi(v)-r\cos\theta u),\\
& b(u,v)=\cos(\varphi(v)-r\cos\theta u)
\end{align*}
for some function $\varphi(v)$. Using the same notations as in the
previous section, this means that we have to solve the system
\begin{align}
& \phi_u=-\frac \kappa 2\sin\theta\cos\theta(F_1\sin\phi-F_2\cos\phi)-2\tau\cos^2\theta,\label{phiu}\\
& \phi_v=a(u,v)\phi_u+b(u,v)\left(\lambda(u,v)-\frac \kappa 2\sin\theta(F_1\cos\phi+F_2\sin\phi)\right),\label{phiv}\\
& (F_1)_u=-\sin\theta\cos\theta\cos\phi\left(1+\frac \kappa 4 (F_1^2+F_2^2)\right),\label{F1u}\\
& (F_1)_v=a(u,v)(F_1)_u+b(u,v)\sin\theta\sin\phi\left(1+\frac \kappa 4 (F_1^2+F_2^2)\right) \label{F1v},\\
& (F_2)_u=-\sin\theta\cos\theta\sin\phi\left(1+\frac \kappa 4 (F_1^2+F_2^2)\right), \label{F2u}\\
& (F_2)_v=a(u,v)(F_2)_u-b(u,v)\sin\theta\cos\phi\left(1+\frac \kappa
4 (F_1^2+F_2^2)\right) \label{F2v},
\end{align}
\begin{align}
& (F_3)_u=-\sin\theta (-\tau F_2 \cos\theta\cos\phi +\tau F_1 \cos\theta\sin\phi -\sin\theta), \label{F3u}\\
& (F_3)_v=a(u,v)(F_3)_u-b(u,v) \tau \sin\theta
(F_2\sin\phi+F_1\cos\phi)\label{F3v}.
\end{align}
By solving \eqref{phiu}, \eqref{F1u} and \eqref{F2u}, we obtain
\begin{align*}
& F_1=\ \frac{\sin 2\theta}{2D(v)}\sin\phi+L(v)\cos(\rho(v)),\\
& F_2=-\frac{\sin 2\theta}{2D(v)}\cos\phi+L(v)\sin(\rho(v)),\\
&\phi=\rho(v)+2\arctan\left(\frac{-A+\sqrt{B^2-A^2}\tan\left(-\frac12\sqrt{B^2-A^2}\
u+C(v)\right)}{B}\right),
\end{align*}
where $D(v)$, $L(v)$, $\rho(v)$ and $C(v)$  are integration
constants and $$A(v)=\frac \kappa 4\sin2\theta\ L(v),\qquad
B(v)=D(v)+\frac \kappa 4\left(\frac{\sin^2
2\theta}{4D(v)}+D(v)L^2(v)\right).$$ Remark here that
$B^2-A^2=r^2\cos^2\theta$ is a strictly positive constant. When we
substitute these solutions in the remaining equations \eqref{phiv},
\eqref{F1v}, \eqref{F2v}, \eqref{F3u} and \eqref{F3v}, calculations
get rather complicated. However, we hope that this partial results
can inspire other mathematicians to construct examples of constant
angle surfaces in a general BCV-space or even to obtain a full
classification.
\end{remark}

\end{document}